\documentclass[a4paper]{article}
\usepackage{typearea}
\typearea{18}
\usepackage{amsmath,amssymb,amsthm}
\usepackage{tabularx}
\usepackage{graphicx}
\usepackage{bm}
\usepackage{float}
\newtheorem{thm}{Theorem}[section]
\newtheorem{prop}[thm]{Proposition}
\newtheorem{lem}[thm]{Lemma}

\newtheorem{con}[thm]{Conjecture}
\theoremstyle{definition}

\newtheorem*{ack*}{Acknowledgement}

  \makeatletter
  \makeatletter
    
    \@addtoreset{equation}{section}
  \makeatother
  \date{}
\title{Rational function approximations of the special function $e^{x}E_{1}(x)$ and applications to irrationality of Euler-Gompertz constant $\delta$}
\author{Naoki Murabayashi \and Hayato Yoshida}
\begin{document}
 \maketitle
\begin{abstract}
In \cite{d4}, we gave a method to construct a continued fraction of the function $F(x):=e^{x}E_{1}(x)$. More precisely we define $F_{1}(x)$ as the reciprocal of $F(x)$ and we inductively define $F_{m}(x)$ as the reciprocal of ``$F_{m-1}(x)$ minus the main term of $F_{m-1}(x)$ at infinity''. We calculated the main term of $F_{m}(x)$ at infinity by using \cite[Proposition 2.1]{d4}. This method is analogous to the regular continued fraction expansion of real numbers. \\
\ \ \ \  In this paper we prove that the continued fraction converges to $F(x)$ for any positive real number $x>0$ by following the proof of that the regular continued fraction of a positive and irrational real number $\alpha$ converges to $\alpha$. Essentially we prove inequalities for $Q_{m}(x)$ (in Theorem 4.1) and inequalities $F_{m}(x)>0$ (in Section 5). In particular, we prove stronger inequalities $\displaystyle\frac{P_{2k}(x)}{Q_{2k}(x)}<F(x)<\displaystyle\frac{P_{2k-1}(x)}{Q_{2k-1}(x)}$ (than $F_{m}(x)>0$) and give two proofs of these. In Section 6, we show an asymptotic relation between $Q_{2k}(x)$ and $Q_{2k-1}(x)$ by using properties of the classical Laguerre polynomial.  In Section 7, we consider Euler-Gompertz constant $\delta$. As far as we know, irrationality of $\delta$ is still an open problem. We construct a sequence of rationals $\displaystyle\frac{A_{i}}{B_{i}}\ (i=1,2,3,\cdots)$ such that $\delta B_{i}-A_{i}$ approaches 0 as $i$ approaches infinity and give a sufficient condition of that $\delta B_{i}-A_{i}\neq 0$ for any positive integer $i$. Therefore, if it is proved that this condition holds, it completes a proof of irrationality of Euler-Gompertz constant $\delta$.
 \end{abstract}
\section*{Keywords}
Continued fractions, Recurrence relation, Euler-Gomperts constant, Rational approximation
\section*{Declarations}
\textbf{Funding} The authors did not receive support from any organization for the submitted work.\\ \\
\textbf{Conflicts of interest/Competing interests} The authors have no conflicts of interest to declare that are relevant to the content of this article.\\ \\
\textbf{Availability of data and material} Not applicable.\\ \\
\textbf{Code availability} Not applicable.\\ \\
\textbf{Authors' contrinutions} All authors contributed to the study conception and design.\\ \\
\textbf{Ethics approval} Not applicable.\\ \\
\textbf{Consent to participate} Not applicable.\\ \\
\textbf{Consent for publication} Not applicable.
\newpage
\section{Introduction}
\ \ \ \ It is well known that the exponential integral $E_{1}(x)$ is defined by $\displaystyle\int_{x}^{\infty} \displaystyle\frac{e^{-t}}{t}dt$. Then we define the special function
$$F(x):=e^{x}E_{1}(x)=\displaystyle\int_{0}^{\infty} \displaystyle\frac{e^{-t} }{t+x} dt.$$
Let $\mathcal{O}$ be the Landau symbol. The following was shown in \cite[Proposition2.1]{d4}
 $$F(x)-\left(\sum_{k=1}^{n} (-1)^{k-1}(k-1)!\left(\frac{1}{x}\right)^{k}\right)=\mathcal{O}\left(\frac{1}{x^{n+1}}\right)\ \ \ \ (x\to \infty)$$
for any positive integer $n$. Using the identity obtained by setting $n=1$, we obtain that
$$\frac{F(x)}{\displaystyle\frac{1}{x}}=1+\frac{1}{x}\cdot \frac{\mathcal{O}\left(\displaystyle\frac{1}{x^{2}}\right)}{\displaystyle\frac{1}{x^{2}}}\rightarrow 1\ \ \ \ (x\to \infty).$$
Then, we see that $F(x)\sim x^{-1}$ at infinity.
 Putting $F_{1}(x):=\displaystyle\frac{1}{F(x)}$, we have that $F_{1}(x)\sim x$ at infinity. Next we put $F_{2}(x):=\displaystyle\frac{1}{F_{1}(x)-x}.$ 
Using the identity obtained by setting $n=1$ (resp. $n=2$) in the denominator (resp. the numerator), we have that
$$F_{1}(x)-x=\displaystyle\frac{1-xF(x)}{F(x)}=\displaystyle\frac{1-x^{2} \mathcal{O}\left(\displaystyle\frac{1}{x^{3}}\right)}{1+x\mathcal{O}\left(\displaystyle\frac{1}{x^{2}}\right)} \rightarrow 1\ \ \ \ (x\to \infty).$$ 
Then, we get that $F_{2}(x)\sim 1$ at infinity. We inductively define
$$F_{m}(x):=\frac{1}{F_{m-1}(x)-(\mbox{main term of}\ F_{m-1}(x)\ \mbox{at\ infinity})}.$$
\ \ \ \ In \cite{d4}, we proved that 
$$F_{2m-1}(x)\sim x,\ F_{2m}(x)\sim \frac{1}{m}\ \ \ \ (x\to \infty)$$
for any positive integer $m$. Thus, we obtain a continued fraction of $F(x)$ at infinity. \\
\ \ \ \ In this paper we prove that this continued fraction converges to $F(x)$. More precisely for any integer $n\ge 1$, we get polynomials $P_{n}(x)$, $Q_{n}(x)$ by truncating the continued fraction of $F(x)$, i.e., 

$$\displaystyle\frac{P_{n}(x)}{Q_{n}(x)}:=\cfrac{1}{c_{1}+\cfrac{1}{c_{2}+\cfrac{1}{c_{3}+\cfrac{1}{ \cfrac{\hspace{8mm} \rotatebox{-22}{$\ddots$} \raisebox{-3.5ex}[1ex]{1}  }{ c_{\scalebox{0.5} {n-2}}+\displaystyle\frac{\hspace{7mm} \raisebox{-1ex}[1ex]1}{c_{\scalebox{0.5} {n-1}}+\displaystyle\frac{\raisebox{-2ex}[1ex]1}{c_{\scalebox{0.5}n}}}}}}}},$$
where 
$$c_{m}:=
\begin{cases}
x\ \ \ \ \mbox{if}\ m:\mbox{odd}\vspace{3mm}\\
\displaystyle\frac{2}{m}\ \ \ \ \mbox{if}\ m:\mbox{even}
\end{cases}$$
\vspace{3mm}for $m=1, 2, \cdots, n.$ In the following for the sake of simplicity, we denote this finite continued fraction by 
$$\displaystyle\frac{1}{c_{1}}{ {} \atop +}\displaystyle\frac{1}{c_{2}}{ {} \atop +}{ {} \atop \cdots} { {} \atop +}\displaystyle\frac{1}{c_{n}}.$$
 By following the proof of that the regular continued fraction of a positive and irrational real number $\alpha$ converges to $\alpha$ (which is reviewed Section 2), we see that for our purpose it is enough to show $Q_{n}(x)>0$, $\displaystyle\lim_{n\to \infty} Q_{n}(x)=\infty$, and $F_{m}(x)>0$ for any real number $x>0$ and positive integers $m$, $n$ (explained in Section 3). By using an explicit expression of $Q_{n}(x)$ obtained in \cite[Theorem 5.1]{d4}, we see that $Q_{n}(x)>0$ (because of that all coefficients of $Q_{n}(x)$ are positive) and show certain inequalities for $Q_{n}(x)$ which imply $\displaystyle\lim_{n \to \infty} Q_{n}(x)=\infty$ (explained in Section 4). In Section 5 we show that
\begin{eqnarray*}
\displaystyle\frac{P_{2k}(x)}{Q_{2k}(x)}<F(x)<\displaystyle\frac{P_{2k-1}(x)}{Q_{2k-1}(x)}\ \ \ \ \mathrm{for}\ k=1,2,\cdots &\Longrightarrow& F_{m}(x)>0\ \ \ \ \mathrm{for}\ m=1,2,\cdots \\
&\Longleftrightarrow& F_{m}(x)>
\begin{cases}
0\ \ \ \ \ \mathrm{if}\ m:\mathrm{odd} \vspace{3mm}\\
\displaystyle\frac{2}{m}\ \ \ \mathrm{if}\ m:\mathrm{even}
\end{cases}
\ \mathrm{for}\ m=1,2,\cdots,
\end{eqnarray*}
and give two proofs of the first inequalities. In the first proof, we need the $k$ th derivative and an explicit expression of $P_{n}(x)$ and $Q_{n}(x).$ On the other hand, in the second proof, it is enough to use the first derivative and recurrence relations of $P_{n}(x)$ and $Q_{n}(x)$. Moreover, we do not need the expression of $c_{m}$ but the identities 
$$\displaystyle\frac{1}{c_{m+1}}=\displaystyle\frac{1}{c_{m-1}}+c^{'}_{m}\ \ \ \ \mathrm{for}\ m=2,3,\cdots$$
satisfied by $\{c_{m}\}$, where $c_{m}^{'}$ denotes the derivative of $c_{m}$ with respect to $x$. We seem that this fact has the potential for an another continued fraction of $F(x)$.\\
\ \ \ \ As a result, we have that
$$F(x)=\displaystyle\lim_{n\to \infty} \displaystyle\frac{P_{n}(x)}{Q_{n}(x)}$$
for any positive real number $x>0$ and we can evaluate error terms
$$\left|F(x)-\displaystyle\frac{P_{n}(x)}{Q_{n}(x)}\right|<
\begin{cases}
\displaystyle\frac{1}{Q_{n}(x)Q_{n-1}(x)}\ \ \ \ \mathrm{if}\ n:\mathrm{even},\vspace{3mm} \\
\displaystyle\frac{1}{\displaystyle\frac{2}{n+1}Q_{n}^{2}(x)+Q_{n}(x)Q_{n-1}(x)}\ \ \ \ \mathrm{if}\ n:\mathrm{odd}.
\end{cases}$$
\ \ \ \ We see the following asymptotic relation between $Q_{2k}(x)$ and $Q_{2k-1}(x)$ for any real number $x>0$:
$$Q_{2k-1}(x)\sim \sqrt{kx}Q_{2k}(x).$$
In order to show this, we use the following property of the classical Laguerre polynomial $L_{k}^{(\alpha)}(x)$:
$$L_{k}^{(\alpha)}(x)=\displaystyle\frac{1}{2\sqrt{\pi}}e^{\frac{x}{2}}(-x)^{-\frac{\alpha}{2}-\frac{1}{4}}k^{\frac{\alpha}{2}-\frac{1}{4}}e^{2\sqrt{-kx}}\left(1+\mathcal{O}\left(\displaystyle\frac{1}{\sqrt{k}}\right)\right)\ \ \ \ (k\to \infty).$$
This relation is introduced in \cite{b2}. Actually, we have that $Q_{2k}(x)=L_{k}^{(0)}(-x)$ and $Q_{2k-1}(x)=kL_{k}^{(-1)}(-x)$ (Lemma 6.1).\\ \\
\vspace{5mm}
\ \ \ \ Euler-Gompertz constant $\delta$ is defined by $\displaystyle\int_{0}^{\infty} e^{-t}\log{(t+1)}dt=\displaystyle\int_{0}^{\infty} \displaystyle\frac{e^{-t}}{t+1}dt=F(1)$. It is well known that Euler-Mascheroni constant $\gamma$ has a integral representation $-\displaystyle\int_{0}^{\infty} e^{-t}\log{t}\ dt.$ These constants have a similar integral representation. In \cite{a1} Alexander Aptekarev proved that at least one of them is irrational. This result was improved in \cite{e5} by Tanguy Rivoal. He proved that at least one of them is transcendental. By Theorem 5.2, we have that
$$\displaystyle\frac{P_{2k}(1)}{Q_{2k}(1)}<\delta<\displaystyle\frac{P_{2k-1}(1)}{Q_{2k-1}(1)}.$$
In Section 7, we construct a sequence of rationals $\displaystyle\frac{A_{i}}{B_{i}}\ (i=1,2,3,\cdots)$ such that $\delta B_{i}-A_{i}$ approaches 0 as $i$ approaches infinity. Since the sequence $\delta\{Q_{m}(1)\}-\{P_{m}(1)\}\ (m=1,2,3,\cdots)$ is bounded (where $\{x\}$ denotes the fractional part of a real number $x$), there exists a convergent subsequence $\delta\{Q_{m_{i}}(1)\}-\{P_{m_{i}}(1)\}\ (i=1,2,3,\cdots)$. Let $\alpha$ be the limit of this. Then we define
$$A_{i}:=[P_{m_{i+1}}(1)]-[P_{m_{i}}(1)],\ \ \ \ B_{i}:=[Q_{m_{i+1}}(1)]-[Q_{m_{i}}(1)]$$
for any positive integer $i$ (where $[x]$ denotes the integer part of a real number $x$). Since $\delta Q_{m_{i}}(1)-P_{m_{i}}(1)$ approaches 0 and $\delta\{Q_{m_{i}}(1)\}-\{P_{m_{i}}(1)\}$ approaches $\alpha$ as $i$ approaches infinity, then we obtain that
$$\delta B_{i}-A_{i}=\delta[Q_{m_{i+1}}(1)]-[P_{m_{i+1}}(1)]-(\delta[Q_{m_{i}}(1)]-[P_{m_{i}}(1)])\to -\alpha+\alpha=0\ \ \ \ (i\to \infty).$$
If $\delta \neq \displaystyle\frac{A_{i}}{B_{i}}$ for any positive integer $i$, Theorem 7.2 implies that Euler-Gompertz constant is irrational. At the end of this paper we give a sufficient condition of that $\delta \neq \displaystyle\frac{A_{i}}{B_{i}}$ for any positive integer $i$.

The continued fraction representation of $F(x)$ is a classical result and there are many related papers. So it is almost impossible to list them all. Therefore thinking that there are already literatures proving that the continued fraction converges to $F(x)$, we have not been able to identify them. One of the purposes of this paper is to clearly describe the origined proof of this convergence in order to apply the continued fraction of $F(1)$ to the proof of the irrationality of the Euler-Gompertz constant.
\section{Review of convergence of the regular continued fraction}
\ \ \ \ Let $\alpha_{0}$ be a positive real number. We put $\alpha_{0}=a_{0}+b_{0}$, where $a_{0}=[\alpha_{0}]$ is the integer part of $\alpha_{0}$ and $b_{0}$ is fractional part of $\alpha_{0}$. If $b_{0}\neq 0$, we can write $b_{0}=\displaystyle\frac{1}{\alpha_{1}}$, with $\alpha_{1}>1$. Hence,
$$\alpha_{0}=a_{0}+\frac{1}{\alpha}.$$
Again with $\alpha_{1}$, if $\alpha_{1}$ is not an integer, we can write $\alpha_{1}=a_{1}+\displaystyle\frac{1}{\alpha_{2}}$, with $a_{1}\ge 1$ and $\alpha_{2}>1$. Therefore,
$$\alpha_{0}=a_{0}+\frac{1}{a_{1}+\displaystyle\frac{1}{\alpha_{2}}}.$$
We can go on the same way as long as $\alpha_{n}$ is not an integer, and we get 
$$\alpha_{0}=a_{0}+\frac{1}{a_{1}}{ {} \atop +} \frac{1}{a_{2}}{ {} \atop +}{ {} \atop \cdots} { {} \atop +}\frac{1}{a_{n}}{ {} \atop +}\frac{1}{\alpha_{n+1}}$$
by using the following algorithm:
$$\alpha_{n}=a_{n}+\frac{1}{\alpha_{n+1}},\ \ \ \ a_{n}=[\alpha_{n}].\eqno(2.1)$$
We now assume that $\alpha_{0}$ is irrational. We consider the regular continued fraction 
$$a_{0}+\displaystyle\frac{1}{a_{1}}{ {} \atop +}\displaystyle\frac{1}{a_{2}} { {} \atop +}{ {} \atop \cdots} { {} \atop +}\displaystyle\frac{1}{a_{n}}{ {} \atop +\cdots},$$
defined by algorithm (2.1), and prove that it converges to $\alpha_{0}$. By \cite[Theorem 3.1]{c3}, we see that 
$$\displaystyle\frac{P_{n}}{Q_{n}}=a_{0}+\displaystyle\frac{1}{a_{1}} { {} \atop +}\displaystyle\frac{1}{a_{2}} { {} \atop +}{ {} \atop \cdots}{ {} \atop +} \displaystyle\frac{1}{a_{n}},$$
where $P_{n}$ and $Q_{n}$ satisfy
$$ \begin{cases}
P_{-1}=1,\ \ Q_{-1}=0,\ \ P_{0}=a_{0},\ \ Q_{0}=1,\\
P_{n}=a_{n}P_{n-1}+P_{n-2},\ \ Q_{n}=a_{n}Q_{n-1}+Q_{n-2}.
\end{cases}$$
We put
\begin{eqnarray*}
\alpha_{0}&=&a_{0}+\frac{1}{a_{1}}{ {} \atop +} \frac{1}{a_{2}}{ {} \atop +}{ {} \atop \cdots} { {} \atop +}\frac{1}{a_{n-1}}{ {} \atop +}\frac{1}{a_{n}} { {} \atop +} \frac{1}{\alpha_{n+1}} \\ \\
&=&a_{0}+\frac{1}{a_{1}}{ {} \atop +} \frac{1}{a_{2}}{ {} \atop +}{ {} \atop \cdots} { {} \atop +}\frac{1}{a_{n-1}}{ {} \atop +}\frac{1}{a_{n}+\displaystyle\frac{1}{\alpha_{n+1}}}=:\displaystyle\frac{P^{'}_{n}}{Q^{'}_{n}}.
\end{eqnarray*}
By recurrence relations, we have that 
\begin{eqnarray*}
P^{'}_{n}&=&\left( a_{n}+\displaystyle\frac{1}{\alpha_{n+1}}\right)P_{n-1}+P_{n-2}\\
&=&a_{n}P_{n-1}+P_{n-2}+\displaystyle\frac{1}{\alpha_{n+1}}P_{n-1}\\
&=&P_{n}+\displaystyle\frac{1}{\alpha_{n+1}}P_{n-1}.
\end{eqnarray*}
Similarly, we have that
$$Q^{'}_{n}=Q_{n}+\displaystyle\frac{1}{\alpha_{n+1}}Q_{n-1}.$$
So, we obtain that
$$\alpha_{0}=\displaystyle\frac{P_{n}+\displaystyle\frac{1}{\alpha_{n+1}}P_{n-1}}{Q_{n}+\displaystyle\frac{1}{\alpha_{n+1}}Q_{n-1}}=\frac{\alpha_{n+1}P_{n}+P_{n-1}}{\alpha_{n+1}Q_{n}+Q_{n-1}}.$$
Since $a_{n}$ are positive integer for every $n\ge1$, $P_{n}$ and $Q_{n}$ are also similar. Moreover, we have that $Q_{n}\ge Q_{n-1}$ since $Q_{n}=a_{n}Q_{n-1}+Q_{n-2}$. Therefore, we obtain that
$$Q_{n}= a_{n}Q_{n-1}+Q_{n-2}\ge a_{n}Q_{n-2}+Q_{n-2}\ge 2Q_{n-2}\ge 2^{2}Q_{n-4}\ge \cdots \ge 
\begin{cases}
2^{[\frac{n}{2}]}Q_{1}\ge 2^{[\frac{n}{2}]}\ \ \mathrm{if}\ n:\mathrm{odd},\vspace{3mm}\\
2^{[\frac{n}{2}]}Q_{0}= 2^{[\frac{n}{2}]}\ \ \mathrm{if}\ n:\mathrm{even}.
\end{cases}$$
Further we see that $P_{n}Q_{n-1}-P_{n-1}Q_{n}=(-1)^{n-1}$ by \cite[Theorem 3.2]{c3} (it can be proved by induction). Therefore, we have that
$$\alpha_{0}-\frac{P_{n}}{Q_{n}}=\frac{\alpha_{n+1}P_{n}+P_{n-1}}{\alpha_{n+1}Q_{n}+Q_{n-1}}-\frac{P_{n}}{Q_{n}}=\frac{-(P_{n}Q_{n-1}-P_{n-1}Q_{n})}{Q_{n}(\alpha_{n+1}Q_{n}+Q_{n-1})}=\frac{(-1)^{n}}{Q_{n}(\alpha_{n+1}Q_{n}+Q_{n-1})}.$$
This equation implies that $\alpha_{0}>\displaystyle\frac{P_{n}}{Q_{n}}$ if $n$ is even, and that $\alpha_{0}<\displaystyle\frac{P_{n}}{Q_{n}}$ if $n$ is odd. Hence, $\alpha_{0}$ always lies between two successive convergents. Moreover, since $Q_{n-1}>0$ and $\alpha_{n+1}>1$, we have that 
$$  \left| \alpha_{0}-\frac{P_{n}}{Q_{n}} \right|<\frac{1}{Q_{n}^{2}}.$$
Since $Q_{n}\ge2^{[\frac{n}{2}]}$, we have that $\displaystyle\lim_{n \to \infty} \displaystyle\frac{P_{n}}{Q_{n}}=\alpha_{0}$.
\section{Sufficient conditions for the continued fraction to converge to $F(x)$}
\ \ \ \ In \cite{4}, we inductively defined $F_{1}(x):=\displaystyle\frac{1}{F(x)}$ and
$$F_{n+1}(x)=\displaystyle\frac{1}{F_{n}(x)-(\mbox{main term of}\ F_{n}(x)\ \mbox{at\ infinity})},\eqno(3.1)$$
for any positive integer $n$ and proved in [2, Theorem 2.1] that
$$\mbox{``main term of }\ F_{n}(x)\ \mbox{at infinity''}=
\begin{cases}
x\ \ \ \ \mbox{if}\ n:\mbox{odd},\vspace{3mm}\\ 
\displaystyle\frac{2}{n}\ \ \ \ \mbox{if}\ n:\mbox{even},
\end{cases}$$
for any positive integer $n$. Since (3.1) is equivalent to 
$$F_{n}(x)=c_{n}+\displaystyle\frac{1}{F_{n+1}(x)},$$
if we apply the argument in Section 2, we obtain recurrence relations:
$$ \begin{cases}
P_{-1}(x)=1,\ \ Q_{-1}(x)=0,\ \ P_{0}(x)=0,\ \ Q_{0}(x)=1,\\
P_{n}(x)=c_{n}P_{n-1}(x)+P_{n-2}(x),\ \ Q_{n}(x)=c_{n}Q_{n-1}(x)+Q_{n-2}(x).
\end{cases}
$$
Then we have that
$$F(x)=\displaystyle\frac{1}{x} { {} \atop +}\displaystyle\frac{1}{1}{ {} \atop +}\displaystyle\frac{1}{x} { {} \atop +}\displaystyle\frac{1}{1/2} { {} \atop +}{ {} \atop \cdots}{ {} \atop +}\displaystyle\frac{1}{c_{n}} { {} \atop +}\displaystyle\frac{1}{F_{n+1}(x)}$$
and 
$$F(x)=\displaystyle\frac{F_{n+1}(x)P_{n}(x)+P_{n-1}(x)}{F_{n+1}(x)Q_{n}(x)+Q_{n-1}(x)}.\eqno(3.2)$$
Hence, we obtain that
$$F(x)-\displaystyle\frac{P_{n}(x)}{Q_{n}(x)}=\displaystyle\frac{(-1)^{n}}{Q_{n}(x)(F_{n+1}(x)Q_{n}(x)+Q_{n-1}(x))}.$$
By recurrence relations, we see that $Q_{n}(x)>0$ for any positive integer $n$ and any positive real number $x$. If $F_{n+1}(x)>0$, we obtain that
$$\left|F(x)-\displaystyle\frac{P_{n}(x)}{Q_{n}(x)}\right|<\displaystyle\frac{1}{Q_{n}(x)Q_{n-1}(x)}.$$ 
\\
\ \ \ \ In order to show that $\displaystyle\lim_{n\to \infty} \displaystyle\frac{P_{n}(x)}{Q_{n}(x)}=F(x)$ for any real number $x>0$, we will show that $F_{n+1}(x)>0$ and $\displaystyle\lim_{n\to \infty} Q_{n}(x)= \infty$ for any real number $x>0$.
\section{Lower bounds of $Q_{n}(x)$}
\ \ \ \ In \cite[Theorem 5.1]{d4}, we gave an explicit expression of $P_{n}(x)$ and $Q_{n}(x)$ respectively. If $n$ is even, $P_{n}(x)$ and $Q_{n}(x)$ are given by
$$P_{2k}(x)=\displaystyle\sum_{i=0}^{k-1} \displaystyle\sum_{l=0}^{k-i-1}\displaystyle\frac{{}_{k} C _{l+i+1}}{(l+i+1)\cdots(l+1)} (-1)^{l}x^{i},\eqno(4.1)$$
$$Q_{2k}(x)=\displaystyle\sum_{i=0}^{k} \displaystyle\frac{ {}_{k} C _{i}}{i!} x^{i},\eqno(4.2)$$
and if $n$ is odd, $P_{n}(x)$ and $Q_{n}(x)$ are given by
$$P_{2k-1}(x)=\displaystyle\sum_{i=0}^{k-1} \displaystyle\sum_{l=0}^{k-i-1} \displaystyle\frac{{}_{k} C _{l+i+1}}{(l+i)\cdots(l+1)} (-1)^{l}x^{i},\eqno(4.3)$$
$$Q_{2k-1}(x)=\displaystyle\sum_{i=0}^{k-1}\displaystyle\frac{ {}_{k} C _{i+1}}{i!}x^{i+1}.\eqno(4.4)$$
In order to show that $\displaystyle\lim_{n\to \infty} Q_{n}(x)= \infty$, we need to find lower bounds of $Q_{2k-1}(x)$ and $Q_{2k}(x)$. For this, we have the following Theorem.
\begin{thm}
For any positive integer $k$ and any positive real number $x$, it holds that
$$Q_{2k}(x)\ge (x+1)^{K},\ \ \ \ Q_{2k-1}(x)\ge x(x+1)^{L},$$
where $K=[\sqrt{k}], L=\left[\displaystyle\frac{-1+\sqrt{1+4k}}{2}\right]$.
\end{thm}
\begin{proof}
We fix $k\ge 1$. Firstly we find a positive integer $K$ such that  $\displaystyle\sum_{i=0}^{k} \displaystyle\frac{ {}_{k} C _{i}}{i!} x^{i} -(x+1)^{K}$ is greater than or equal to 0 for $x>0$. It is enough that all coefficients of this polynomial are non-negative. In other words we find $K$ such that $k\ge K$ and $\displaystyle\frac{ {}_{k} C _{i}}{i!} \ge{}_{K} C _{i}$ for $0\le i\le K $. If $i=0$, it is trivial. So we consider cases of $1\le i\le K$. Multiplying both sides of the inequality by $(i!)^{2}$, then we have that
$$k(k-1)\cdots (k-i+1)\ge i!K(K-1)\cdots (K-i+1).$$
It is sufficient that $K$ satisfies the following conditions:
$$k\ge iK,\ \ \ \ k-1\ge (i-1)(K-1),\ \ \ \ \cdots,\ \ \ \ k-i+1\ge K-i+1$$
for $1\le i\le K$. Letting $i=K$ in the first inequality, we see that $\sqrt{k}\ge K$. Conversely, $1\le K\le \sqrt{k}$ imply $k-j\ge (i-j)(K-j)$ for $0\le j\le i-1$ and $1\le i \le K$. In fact, we have that
\begin{eqnarray*}
(i-j)(K-j)&\le&K(K-j)\\
&=&K^{2}-Kj\\
&\le&K^{2}-1\cdot j\\
&\le& k-j.
\end{eqnarray*}
 Putting $K=[\sqrt{k}]$, we obtain that
$$Q_{2k}(x)\ge (x+1)^{K}$$
for any positive real number $x$.\\
\ \ \ \ We next find a positive integer $L$ such that $\displaystyle\sum_{i=0}^{k-1} \displaystyle\frac{ {}_{k} C _{i+1}}{i!} x^{i+1}-x(x+1)^{L}$ is greater than or equal to 0 for $x>0$. Similarly, we find $L$ such that $k\ge L+1$ and $\displaystyle\frac{ {}_{k} C _{i+1}}{i!} \ge{}_{L} C _{i}$ for $0\le i\le L$. If $i=0$, It is also trivial. So we consider cases of $1\le i\le L\ (\le k-1)$. Multiplying both sides of the inequality by $i!(i+1)!$, then we have that
$$k(k-1)\cdots (k-i+1)(k-i)\ge (i+1)!L(L-1)\cdots (L-i+1)$$
for $1\le i\le L$. It is also sufficient that $L$ satisfies the following conditions:
$$k\ge (i+1)L,\ \ \ \ k-1\ge i(L-1),\ \ \ \ \cdots,\ \ \ \ k-i+1\ge 2(L-i+1),\ \ \ \ k-i\ge1.$$
The last inequality is trivial. Letting $i=L$ in the first inequality, we see that $\displaystyle\frac{-1+\sqrt{4k+1}}{2}\ge L$. Conversely,  $1\le L\le\displaystyle\frac{-1+\sqrt{4k+1}}{2}$ imply $k-j\ge (i-j+1)(L-j)$ for $0\le j\le i-1$ and $1\le i\le L$. In fact, we have that
\begin{eqnarray*}
(i-j+1)(L-j)&\le& (L+1)(L-j)\\
&=&L^{2}+(1-j)L-j\\
&\le& L^{2}+L-j\\
&\le&\displaystyle\frac{1+2k-\sqrt{4k+1}}{2}+\displaystyle\frac{-1+\sqrt{4k+1}}{2}-j\\
&=&k-j.
\end{eqnarray*}
Putting $L=\left[\displaystyle\frac{-1+\sqrt{4k+1}}{2}\right]$, we obtain that
$$Q_{2k-1}(x)\ge x(x+1)^{L}.$$
These complete a proof of Theorem 4.1.
\end{proof}
\vspace{5mm}
By Theorem 4.1, we see that
$$\lim_{n\to \infty}Q_{n}(x)=\infty$$
for any positive real number $x$.
\section{Two proofs of evaluations of $F_{n}(x)$}
\ \ \ \ In this section, we show that $F_{n}(x)>0$ for any positive real number $x$. Firstly, we introduce the following Lemma.
\begin{lem}
We consider the following conditios:\vspace{3mm}\\
\ \ \ \ \ \ $(i)$\ \ $\displaystyle\frac{P_{2k}(x)}{Q_{2k}(x)}<F(x)<\displaystyle\frac{P_{2k-1}(x)}{Q_{2k-1}(x)}\ \ (k=1,2,\cdots);$\vspace{3mm}\\
\ \ \ \ \ $(ii)$\ \ $F_{n}(x)>0\ \ (n=1,2,\cdots);$\vspace{3mm}\\
\ \ \ \ $(iii)$\ \ \ \ $F_{n}(x)>
\begin{cases}
0\ \ \ \ \ \mathrm{if}\ n:\mathrm{odd}, \vspace{3mm}\\
\displaystyle\frac{2}{n}\ \ \ \ \mathrm{if}\ n:\mathrm{even}
\end{cases}
\ \ (n=1,2,\cdots).$\vspace{3mm}\\ 
Then, we have that the condition $(i)$ implies the condition $(ii)$ and the conditions $(ii)$ and $(iii)$ are equivalent.
\end{lem}
\begin{proof}
$(i)\Longrightarrow (ii)$. By the equation (3.2), we see that
$$(F(x)Q_{n}(x)-P_{n}(x))F_{n+1}(x)=P_{n-1}(x)-F(x)Q_{n-1}(x)$$
for any positive integer $n$. If $n=2k$, since $F(x)Q_{2k}(x)-P_{2k}(x)>0$ and $P_{2k-1}(x)-F(x)Q_{2k-1}(x)>0$, we obtain that
$$F_{2k+1}(x)=\displaystyle\frac{P_{2k-1}(x)-F(x)Q_{2k-1}(x)}{F(x)Q_{2k}(x)-P_{2k}(x)}>0$$
and if $n=2k+1$, we also have that
$$F_{2k+2}(x)=\displaystyle\frac{P_{2k}(x)-F(x)Q_{2k}(x)}{F(x)Q_{2k+1}(x)-P_{2k+1}(x)}>0$$
for any positive integer $k$. If $n=1$, $F_{1}(x)=\displaystyle\frac{1}{F(x)}>0$ and if $n=2$, $F_{2}(x)=\displaystyle\frac{1}{F(x)Q_{1}(x)-P_{1}(x)}>0$.\\
\ \ \ \ $(ii)\Longrightarrow (iii)$. By the assumption, we have that
$$F_{2k}(x)=\displaystyle\frac{1}{F_{2k-1}(x)-x}>0,\ \ \ \ F_{2k+1}(x)=\displaystyle\frac{1}{F_{2k}(x)-\displaystyle\frac{1}{k}}>0.$$
So we obtain that $F_{2k-1}(x)>x>0$ and $F_{2k}(x)>\displaystyle\frac{1}{k}$ for any positive integer $k$. Then we see that
$$F_{n}(x)>
\begin{cases}
0\ \ \ \ \ \mathrm{if}\ n:\mathrm{odd}, \vspace{3mm}\\
\displaystyle\frac{2}{n}\ \ \ \ \mathrm{if}\ n:\mathrm{even}
\end{cases}
$$
for any positive integer $n$.\\
\ \ \ \ $(iii)\Longrightarrow (ii)$. It is trivial.\\
\ \ \ \ These complete a proof of Lemma 5.1.
\end{proof}
\vspace{5mm}
By Lemma 5.1, if we prove the condition $(i)$, then we obtain the condition $(iii)$ and this implies the evaluation of error terms stated at the end of Introduction. We give below two proofs of the condition $(i)$.
\begin{thm}
It holds the condition $(i)$ in Lemma 5.1.
\end{thm}
\begin{proof}[The first proof] We consider the $k$ th derivative of $x^{k}F(x)$. Then, we obtain that
$$\{x^{k}F(x)\}^{(k)}=\displaystyle\sum_{i=0}^{k} {}_{k} C _{i} \displaystyle\frac{k!}{i!}x^{i} F^{(i)}(x),\eqno(5.1)$$
where $F^{(0)}(x):=F(x)$. Since $F(x)=\displaystyle\int_{0}^{\infty} \displaystyle\frac{e^{-t} }{t+x} dt$, we see that
$$F^{(i)}(x)=(-1)^{i}i! \displaystyle\int_{0}^{\infty} \displaystyle\frac{e^{-t} }{(t+x)^{i+1}} dt.\eqno(5.2)$$
Substituting this equality into the right hand side of the equality (5.1), we have that
\begin{eqnarray*}
\{x^{k}F(x)\}^{(k)}&=&\displaystyle\sum_{i=0}^{k} {}_{k} C _{i} \displaystyle\frac{k!}{i!}x^{i}(-1)^{i}i! \displaystyle\int_{0}^{\infty} \displaystyle\frac{e^{-t} }{(t+x)^{i+1}} dt\\
&=&k!\displaystyle\int_{0}^{\infty} \displaystyle\frac{e^{-t}}{t+x} \displaystyle\sum_{i=0}^{k}   {}_{k} C _{i} \displaystyle\frac{(-x)^{i}}{(t+x)^{i}} dt.
\end{eqnarray*}
Since $\displaystyle\sum_{i=0}^{k}   {}_{k} C _{i} \displaystyle\frac{(-x)^{i}}{(t+x)^{i}}=\left(1-\displaystyle\frac{x}{t+x}\right)^{k}=\left(\displaystyle\frac{t}{t+x}\right)^{k}$, then we obtain that
$$\{x^{k}F(x)\}^{(k)}=k!\displaystyle\int_{0}^{\infty} \displaystyle\frac{e^{-t}}{t+x} \left(\displaystyle\frac{t}{t+x}\right)^{k}dt=k!\displaystyle\int_{0}^{\infty} \displaystyle\frac{e^{-t}t^{k}}{(t+x)^{k+1}}dt>0.$$
Using an integration by parts in the right hand side of the equality (5.2), then we see that
\begin{eqnarray*}
F^{(i)}(x)&=&(-1)^{i}i!\left(\left[-\displaystyle\frac{e^{-t}}{i(t+x)^{i}}\right]_{0}^{\infty}-\displaystyle\int_{0}^{\infty} \displaystyle\frac{e^{-t}}{i(t+x)^{i}}dt\right)\\
&=&(-1)^{i-1}(i-1)!\displaystyle\int_{0}^{\infty} \displaystyle\frac{e^{-t}}{(t+x)^{i}}dt-\displaystyle\frac{(-1)^{i-1}(i-1)!}{x^{i}}\\
&=&F^{(i-1)}(x)-\displaystyle\frac{(-1)^{i-1}(i-1)!}{x^{i}}.
\end{eqnarray*}
Therefore, we obtain that
$$F^{(i)}(x)=F(x)-\displaystyle\sum_{l=1}^{i} \displaystyle\frac{(-1)^{l-1}(l-1)!}{x^{l}}.\eqno(5.3)$$
Substituting this equality into the right hand side of the equality (5.1), we obtain that
\begin{eqnarray*}
\{x^{k}F(x)\}^{(k)}&=&\displaystyle\sum_{i=0}^{k} {}_{k} C _{i} \displaystyle\frac{k!}{i!}x^{i} F^{(i)}(x)\\
&=&k!F(x)+k!\displaystyle\sum_{i=1}^{k}   \displaystyle\frac{{}_{k} C _{i}}{i!}x^{i} \left(F(x)-\displaystyle\sum_{l=1}^{i} \displaystyle\frac{(-1)^{l-1}(l-1)!}{x^{l}}\right)\\
&=&k!F(x)\displaystyle\sum_{i=0}^{k} \displaystyle\frac{{}_{k} C _{i}}{i!}x^{i}-k!\displaystyle\sum_{i=1}^{k} \displaystyle\sum_{l=1}^{i}  {}_{k} C _{i}\displaystyle\frac{(-1)^{l-1}(l-1)!}{i!}x^{i-l}\\
&=&k!F(x)\displaystyle\sum_{i=0}^{k} \displaystyle\frac{{}_{k} C _{i}}{i!}x^{i}-k!\displaystyle\sum_{i=0}^{k-1} \displaystyle\sum_{l=0}^{i}  {}_{k} C _{i+1}\displaystyle\frac{(-1)^{l}l!}{(i+1)!}x^{i-l}\\
&=&k!F(x)\displaystyle\sum_{i=0}^{k} \displaystyle\frac{{}_{k} C _{i}}{i!}x^{i}-k!\displaystyle\sum_{i=0}^{k-1} \displaystyle\sum_{l=0}^{k-i-1}  {}_{k} C _{i+l+1}\displaystyle\frac{(-1)^{l}l!}{(i+l+1)!}x^{i}\\
&=&k!F(x)\displaystyle\sum_{i=0}^{k} \displaystyle\frac{{}_{k} C _{i}}{i!}x^{i}-k!\displaystyle\sum_{i=0}^{k-1} \displaystyle\sum_{l=0}^{k-i-1}\displaystyle\frac{{}_{k} C _{l+i+1}}{(l+i+1)\cdots(l+1)}(-1)^{l}x^{i}>0.
\end{eqnarray*}
Therefore, we obtain that
$$F(x)>\displaystyle\frac{\displaystyle\sum_{i=0}^{k-1} \displaystyle\sum_{l=0}^{k-i-1}\displaystyle\frac{{}_{k} C _{l+i+1}}{(l+i+1)\cdots(l+1)} (-1)^{l}x^{i}}{\displaystyle\sum_{i=0}^{k} \displaystyle\frac{ {}_{k} C _{i}}{i!} x^{i}}=\displaystyle\frac{P_{2k}(x)}{Q_{2k}(x)}$$
by the equalities (4.1) and (4.2).\\
\ \ \ \ We next consider multiplying the $k$ th derivative of $x^{k-1}F(x)$ by $x$. Then we obtain that
$$x\{x^{k-1}F(x)\}^{(k)}=x\displaystyle\sum_{i=1}^{k}  {}_{k} C _{i} \displaystyle\frac{(k-1)!}{(i-1)!}x^{i-1}F^{(i)}(x).\eqno(5.4)$$
Substituting the equality (5.2) into the right hand side of this equality (5.4), we have that
\begin{eqnarray*}
x\{x^{k-1}F(x)\}^{(k)}&=&\displaystyle\sum_{i=1}^{k}  {}_{k} C _{i} \displaystyle\frac{(k-1)!}{(i-1)!}x^{i}(-1)^{i}i! \displaystyle\int_{0}^{\infty} \displaystyle\frac{e^{-t} }{(t+x)^{i+1}} dt\\
&=&(k-1)!\displaystyle\int_{0}^{\infty} \displaystyle\frac{e^{-t}}{t+x} \displaystyle\sum_{i=1}^{k}   {}_{k} C _{i} \displaystyle\frac{i(-x)^{i}}{(t+x)^{i}} dt.
\end{eqnarray*}
By calculating the derivative of the both sides of the identity $\left(1-\displaystyle\frac{x}{t+x}\right)^{k}=\displaystyle\sum_{i=0}^{k}   {}_{k} C _{i} \displaystyle\frac{(-x)^{i}}{(t+x)^{i}}$ with respect to $x$, we obtain that
$$k\left(1-\displaystyle\frac{x}{t+x}\right)^{k-1}\displaystyle\frac{-t}{(t+x)^{2}}=\displaystyle\sum_{i=1}^{k}   {}_{k} C _{i} \displaystyle\frac{i(-x)^{i-1}}{(t+x)^{i-1}}\cdot \displaystyle\frac{-t}{(t+x)^{2}}.$$
Dividing the both sides by $-\displaystyle\frac{t}{(t+x)^{2}}$ and multiplying them by $-\displaystyle\frac{x}{t+x}$, we see that
$$\displaystyle\sum_{i=1}^{k}   {}_{k} C _{i} \displaystyle\frac{i(-x)^{i}}{(t+x)^{i}}=-\displaystyle\frac{kxt^{k-1}}{(t+x)^{k}}.$$
Then, we have that
$$x\{x^{k-1}F(x)\}^{(k)}=(k-1)!\displaystyle\int_{0}^{\infty} \displaystyle\frac{e^{-t}}{t+x} \left(-\displaystyle\frac{kxt^{k-1}}{(t+x)^{k}} \right)dt=-xk!\displaystyle\int_{0}^{\infty} \displaystyle\frac{e^{-t}t^{k-1}}{(t+x)^{k+1}}dt<0.$$
Substituting the equality (5.3) into the right hand side of the equality (5.4), we obtain that
\begin{eqnarray*}
x\{x^{k-1}F(x)\}^{(k)}&=&\displaystyle\sum_{i=1}^{k}  {}_{k} C _{i} \displaystyle\frac{(k-1)!}{(i-1)!}x^{i}\left(F(x)-\displaystyle\sum_{l=1}^{i} \displaystyle\frac{(-1)^{l-1}(l-1)!}{x^{l}}\right)\\
&=&(k-1)!F(x)\displaystyle\sum_{i=0}^{k-1}  \displaystyle\frac{ {}_{k} C _{i+1}}{i!}x^{i+1}-(k-1)!\displaystyle\sum_{i=1}^{k}\displaystyle\sum_{l=1}^{i} {}_{k} C _{i} \displaystyle\frac{(-1)^{l-1}(l-1)!}{(i-1)!}x^{i-l}\\
&=&(k-1)!F(x)\displaystyle\sum_{i=0}^{k-1}  \displaystyle\frac{ {}_{k} C _{i+1}}{i!}x^{i+1}-(k-1)!\displaystyle\sum_{i=0}^{k-1}\displaystyle\sum_{l=0}^{i} {}_{k} C _{i+1} \displaystyle\frac{(-1)^{l}l!}{i!}x^{i-l}\\
&=&(k-1)!F(x)\displaystyle\sum_{i=0}^{k-1}  \displaystyle\frac{ {}_{k} C _{i+1}}{i!}x^{i+1}-(k-1)!\displaystyle\sum_{i=0}^{k-1}\displaystyle\sum_{l=0}^{k-i-1} {}_{k} C _{i+l+1} \displaystyle\frac{(-1)^{l}l!}{(i+l)!}x^{i}\\
&=&(k-1)!F(x)\displaystyle\sum_{i=0}^{k-1} \displaystyle\frac{ {}_{k} C _{i+1}}{i!}x^{i+1}-(k-1)!\displaystyle\sum_{i=0}^{k-1}\displaystyle\sum_{l=0}^{k-i-1}  \displaystyle\frac{{}_{k} C _{i+l+1}}{(i+l)\cdots (l+1)}(-1)^{l}x^{i}<0.
\end{eqnarray*}
Therefore, we have that
$$F(x)<\displaystyle\frac{\displaystyle\sum_{i=0}^{k-1}\displaystyle\sum_{l=0}^{k-i-1}  \displaystyle\frac{{}_{k} C _{i+l+1}}{(i+l)\cdots (l+1)}(-1)^{l}x^{i}}{\displaystyle\sum_{i=0}^{k-1}  {}_{k} C _{i+1} \displaystyle\frac{1}{i!}x^{i+1}}=\displaystyle\frac{P_{2k-1}(x)}{Q_{2k-1}(x)}$$
by the equalities (4.3) and (4.4).
\end{proof}
\vspace{5mm}
We define the function $f_{m}(x)$ by
$$(-1)^{m}\left(e^{-x}F(x)-e^{-x}\displaystyle\frac{P_{m}(x)}{Q_{m}(x)}\right)$$
for any positive integer $m$. If we see $f_{m}(x)>0$ for any positive integer $m$, we have that
$$\displaystyle\frac{P_{2k}(x)}{Q_{2k}(x)}<F(x)<\displaystyle\frac{P_{2k-1}(x)}{Q_{2k-1}(x)}$$
for any positive integer $k$. In order to show this, in the first proof we need the $k$ th derivative and an explicit expression of $P_{n}(x)$ and $Q_{n}(x)$, however in the second proof it is enough to use the derivative of $f_{m}(x)$ and recurrence relations of $P_{n}(x)$ and $Q_{n}(x)$. In order to show that $f_{m}(x)>0$, we prepare the following lemma.
\begin{lem}
Let $f(x)$ be a differentiable function defined on $\mathbb{R}_{>0}$. We assume that the limit of the function $f(x)$ equals 0 as $x$ approaches infinity and the derivative of $f(x)$, denoted by $f^{'}(x)$ satisfies the condition:
$$f^{'}(x)<0\ \ \ \ \mbox{for}\ \forall x>0.$$ 
Then we have that 
$$f(x)>0\ \ \ \ \mbox{for}\ \forall x>0.$$
\end{lem}
\begin{proof}
We assume that there exists $x_{0}>0$ such that $f(x_{0})\le 0$. Since $f^{'}(x)<0$ for any positive real number $x$, the function $f(x)$ is strictly decreasing. Therefore, we have that 
$$f(x_{0}+1)<f(x_{0})\le 0.$$
Hence putting $\varepsilon:=-f(x_{0}+1)$, we obtain that
$$|f(x)|\ge \varepsilon\ \ \ \ \mbox{for}\ \forall x>x_{0}+1.$$
This contradicts $\displaystyle\lim_{x\to \infty} f(x)=0$.
\end{proof}
\vspace{5mm}
By Lemma 5.3, for our purpose it is enough to show that $\displaystyle\lim_{x\to \infty} f_{m}(x)=0$ and $f^{'}_{m}(x)<0$ for any positive real number $x$. Since $\displaystyle\lim_{x\to \infty} F(x)=0$, we see that 
$$\displaystyle\lim_{x\to \infty} f_{m}(x)=\displaystyle\lim_{x\to \infty} (-1)^{m}\left(e^{-x}F(x)-e^{-x}\displaystyle\frac{P_{m}(x)}{Q_{m}(x)}\right)=0.$$ 
The derivative of the function $f_{m}(x)$ is calculated as the following:
\begin{eqnarray*}
f^{'}_{m}(x)&=&(-1)^{m}\left(-e^{-x}F(x)+e^{-x}F^{'}(x)+e^{-x}\displaystyle\frac{P_{m}(x)}{Q_{m}(x)}-e^{-x}\displaystyle\frac{P^{'}_{m}(x)Q_{m}(x)-P_{m}(x)Q^{'}_{m}(x)}{Q^{2}_{m}(x)}\right)\\
&=&(-1)^{m}e^{-x}\displaystyle\frac{Q^{2}_{m}(x)(F^{'}(x)-F(x))+P_{m}(x)Q_{m}(x)-P^{'}_{m}(x)Q_{m}(x)+P_{m}(x)Q^{'}_{m}(x)}{Q^{2}_{m}(x)}.
\end{eqnarray*}
By the equation (5.3) with $i=1$, we see that $F^{'}(x)=F(x)-\displaystyle\frac{1}{x}$. Then we obtain that
\begin{eqnarray*}
f^{'}_{m}(x)&=&(-1)^{m}e^{-x}\displaystyle\frac{-\displaystyle\frac{1}{x}Q^{2}_{m}(x)+P_{m}(x)Q_{m}(x)-P^{'}_{m}(x)Q_{m}(x)+P_{m}(x)Q^{'}_{m}(x)}{Q^{2}_{m}(x)}\\ \\
&=&(-1)^{m+1}e^{-x}\displaystyle\frac{Q^{2}_{m}(x)-xP_{m}(x)Q_{m}(x)+xP^{'}_{m}(x)Q_{m}(x)-xP_{m}(x)Q^{'}_{m}(x)}{xQ^{2}_{m}(x)}.
\end{eqnarray*}
In order to calculate the numerator, we define
$$\beta_{m}(x):=Q^{2}_{m}(x)-xP_{m}(x)Q_{m}(x)+xP^{'}_{m}(x)Q_{m}(x)-xP_{m}(x)Q_{m}^{'}(x)$$
for $m\ge 1$. Substituting $P_{m}(x)=c_{m}P_{m-1}(x)+P_{m-2}(x)$ and $Q_{m}(x)=c_{m}Q_{m-1}(x)+Q_{m-2}(x)$ into the right hand side of $\beta_{m}(x)$, we obtain the following the equation under a little complicated calculation:
\begin{eqnarray*}
\beta_{m}(x)&=&c_{m}^{2}\beta_{m-1}(x)+\beta_{m-2}(x)+c_{m} \{2Q_{m-1}(x)Q_{m-2}(x)-xP_{m-1}(x)Q_{m-2}(x)-xP_{m-2}(x)Q_{m-1}(x)\\
& &+xP^{'}_{m-1}(x)Q_{m-2}(x)+xP^{'}_{m-2}(x)Q_{m-1}(x)-xP_{m-1}(x)Q^{'}_{m-2}(x)-xP_{m-2}(x)Q^{'}_{m-1}(x)\}\\
& &+c^{'}_{m}x(P_{m-1}(x)Q_{m-2}(x)-P_{m-2}(x)Q_{m-1}(x)),
\end{eqnarray*}
where $c_{m}^{'}$ denotes the derivative of $c_{m}$ with respect to $x$. Setting
$$\gamma_{m}(x):=2Q_{m}(x)Q_{m-1}(x)-xP_{m}(x)Q_{m-1}(x)-xP_{m-1}(x)Q_{m}(x)+xP^{'}_{m}(x)Q_{m-1}(x)$$
$$+xP^{'}_{m-1}(x)Q_{m}(x)-xP_{m}(x)Q^{'}_{m-1}(x)-xP_{m-1}(x)Q^{'}_{m}(x)$$ 
for $m\ge 2$ and using $P_{m-1}(x)Q_{m-2}(x)-P_{m-2}(x)Q_{m-1}(x)=(-1)^{m}$, we have that
$$\beta_{m}(x)=c^{2}_{m}\beta_{m-1}(x)+\beta_{m-2}(x)+c_{m}\gamma_{m-1}(x)+c_{m}^{'}(-1)^{m}x.$$
A similar calculation leads to the following equation:
$$\gamma_{m}(x)=\gamma_{m-1}(x)+2c_{m}\beta_{m-1}(x).$$
Since $P_{1}(x)=1$, $P_{2}(x)=1$, $Q_{1}(x)=x$, and $Q_{2}(x)=x+1$, we see that $\beta_{1}(x)=-x$, $\beta_{2}(x)=1$, and $\gamma_{2}(x)=-x$. Summarizing the above, the following two recurrence relations are obtained:
$$\begin{cases}
\beta_{m}(x)=c^{2}_{m}\beta_{m-1}(x)+\beta_{m-2}(x)+c_{m}\gamma_{m-1}(x)+c_{m}^{'}(-1)^{m}x,\\
\gamma_{m}(x)=\gamma_{m-1}(x)+2c_{m}\beta_{m-1}(x),
\end{cases}\ \ \ \ c_{m}=
\begin{cases}
x\ \ \ \ \mbox{if}\ m:\mbox{odd},\vspace{3mm} \\ 
\displaystyle\frac{2}{m}\ \ \ \ \mbox{if}\ m:\mbox{even},
\end{cases}$$
for any integer $m$ greater than or equal to 3 with initial conditions $\beta_{1}(x)=-x$, $\beta_{2}(x)=1$, and $\gamma_{2}(x)=-x$. 
\begin{thm}
Let $m$ be any integer grater than or equal to 1. it holds that
$$\beta_{m}(x)=\displaystyle\frac{(-1)^{m}}{c_{m+1}}x,\ \ \ \ \gamma_{m+1}(x)=(-1)^{m}x.$$
\end{thm}
\begin{proof}
We prove these by induction. In the base case, these are true by the initial conditions. Let $k$ be an integer grater than or equal to 3 and assume that 
$$\beta_{k-2}(x)=\displaystyle\frac{(-1)^{k-2}}{c_{k-1}}x,\ \ \ \ \beta_{k-1}(x)=\displaystyle\frac{(-1)^{k-1}}{c_{k}}x,\ \ \ \ \gamma_{k-1}(x)=(-1)^{k-2}x.$$
It is enough to prove the equations about $\beta_{k}(x)$ and $\gamma_{k}(x)$. By recurrence relations and assumptions, we obtain that
\begin{eqnarray*}
\beta_{k}(x)&=&c^{2}_{k}\beta_{k-1}(x)+\beta_{k-2}(x)+c_{k}\gamma_{k-1}(x)+c_{k}^{'}(-1)^{k}x\\ \\
&=&c^{2}_{k}\displaystyle\frac{(-1)^{k-1}}{c_{k}}x+\displaystyle\frac{(-1)^{k-2}}{c_{k-1}}x+c_{k}(-1)^{k-2}x+c_{k}^{'}(-1)^{k}x\\ \\
&=&c_{k}(-1)^{k-1}x+\displaystyle\frac{(-1)^{k-2}}{c_{k-1}}x+c_{k}(-1)^{k-2}x+c_{k}^{'}(-1)^{k}x\\ \\
&=&(-1)^{k}x\left(\displaystyle\frac{1}{c_{k-1}}+c_{k}^{'}\right)
\end{eqnarray*}
and
\begin{eqnarray*}
\gamma_{k}(x)&=&\gamma_{k-1}(x)+2c_{k}\beta_{k-1}(x)\\ \\
&=& (-1)^{k-2}x+2c_{k}\displaystyle\frac{(-1)^{k-1}}{c_{k}}x\\ \\
&=&(-1)^{k-2}x+2(-1)^{k-1}x\\ \\
&=&(-1)^{k-1}x.
\end{eqnarray*}
If $k$ is even, 
$$\displaystyle\frac{1}{c_{k-1}}+c_{k}^{'}=\displaystyle\frac{1}{x}=\displaystyle\frac{1}{c_{k+1}}$$
and if $k$ is odd, 
$$\displaystyle\frac{1}{c_{k-1}}+c_{k}^{'}=\displaystyle\frac{k-1}{2}+1=\displaystyle\frac{k+1}{2}=\displaystyle\frac{1}{c_{k+1}}.$$
This implies $\beta_{k}=\displaystyle\frac{(-1)^{k}}{c_{k+1}}x.$ These completes a proof of Theorem.
\end{proof}
\vspace{5mm}
Now we are ready to give the second proof of Theorem 5.2.
\begin{proof}[The second proof]
By Theorem 5.4, we obtain that for any positive integer $m$
$$f_{m}^{'}(x)=(-1)^{m+1}e^{-x}\displaystyle\frac{\beta_{m}(x)}{xQ_{m}^{2}(x)}=
\begin{cases}
-e^{-x}\displaystyle\frac{k}{Q_{2k-1}^{2}(x)}<0\ \ \ \ \mbox{if}\ m=2k-1,\\ \\
-e^{-x}\displaystyle\frac{1}{xQ_{2k}^{2}(x)}<0\ \ \ \ \ \mbox{if}\ m=2k.
\end{cases}
$$
By applying Lemma 5.3 to $f_{m}(x)$, we see that for any positive integer $m$,
$$f_{m}(x)=(-1)^{m}\left(e^{-x}F(x)-e^{-x}\displaystyle\frac{P_{m}(x)}{Q_{m}(x)}\right)>0.$$
Therefore, we have that for any positive integer $k$,
 $$\displaystyle\frac{P_{2k}(x)}{Q_{2k}(x)}<F(x)<\displaystyle\frac{P_{2k-1}(x)}{Q_{2k-1}(x)}.$$
 \end{proof}

\section{The relationship between $Q_{m}(x)$ and the Laguerre polynomial}
\ \ \ \ We now consider Laguerre polynomials. The classical Laguerre polynomial $L^{(\alpha)}_{k}(x)$ is given by
$$L^{(\alpha)}_{k}(x):=\sum_{i=0}^{k} \displaystyle\frac{{}_{k+\alpha} C _{k-i}}{i!} (-x)^{i}\ \ \ \ (k=0,1,2,\cdots)$$
for any real numbers $\alpha$ and $x$, where ${}_{k+\alpha} C _{k-i}$ is defined by
$$\displaystyle\frac{1}{(k-i)!}\displaystyle\prod_{j=0}^{k-i-1}(k+\alpha-j)$$
if $i<k$ and 1 if $i=k$. It is well known that the classical Laguerre polynomial $L^{(\alpha)}_{k}(x)$ satisfies the recurrence relation
$$L_{k}^{(\alpha)}(x)-L_{k-1}^{(\alpha)}(x)=L_{k}^{(\alpha-1)}(x).$$
Then we obtain the following lemma.
\begin{lem}
For any integer $k$ greater than or equal to 0, it holds that
$$Q_{2k}(x)=L_{k}^{(0)}(-x),$$
and for any integer $k$ greater than or equal to 1, it holds that
$$Q_{2k-1}(x)=kL_{k}^{(-1)}(-x).$$
\end{lem}
\begin{proof}
We easily see that 
$$Q_{2k}(x)=\sum_{i=0}^{k} \displaystyle\frac{{}_{k} C _{i}}{i!}x^{i}=\sum_{i=0}^{k} \displaystyle\frac{{}_{k} C _{k-i}}{i!}x^{i}=L_{k}^{(0)}(-x).$$
Here, we see that
\begin{eqnarray*}
k(Q_{2k}(x)-Q_{2k-2}(x))&=&k\left(\sum_{i=0}^{k} \displaystyle\frac{{}_{k} C _{i}}{i!}x^{i}-\sum_{i=0}^{k-1} \displaystyle\frac{{}_{k-1} C _{i}}{i!}x^{i}\right)\\ \\
&=&k\left(\displaystyle\frac{x^{k}}{k!}+\sum_{i=1}^{k-1} \displaystyle\frac{1}{i!}({}_{k} C _{i}-{}_{k-1} C _{i})x^{i}\right)\\ \\
&=&k\left(\displaystyle\frac{x^{k}}{k!}+\sum_{i=1}^{k-1} \displaystyle\frac{{}_{k-1} C _{i-1}}{i!}x^{i}\right)\\ \\
&=&k\sum_{i=1}^{k}\displaystyle\frac{{}_{k-1} C _{i-1}}{i!}x^{i}\\ \\
&=&\sum_{i=1}^{k} \displaystyle\frac{k}{i} {}_{k-1} C_{i-1} \displaystyle\frac{x^{i}}{(i-1)!}\\ \\
&=&\sum_{i=0}^{k-1} \displaystyle\frac{{}_{k} C _{i+1}}{i!}x^{i+1}=Q_{2k-1}(x),
\end{eqnarray*}
where we use properties of the binomial coefficient ${}_{k} C _{i}-{}_{k-1} C _{i}={}_{k-1} C _{i-1}$ and $\displaystyle\frac{k}{i}{}_{k-1} C _{i-1}={}_{k} C _{i}$. Therefore, by the recurrence relation of $L_{k}^{(\alpha)}$, we have that
$$Q_{2k-1}(x)=kL_{k}^{(-1)}(-x).$$
\end{proof}
\vspace{5mm}
By \cite[Lemma 2.1]{b2}, we see that
$$L_{k}^{(\alpha)}(x)=\displaystyle\frac{1}{2\sqrt{\pi}}e^{\frac{x}{2}}(-x)^{-\frac{\alpha}{2}-\frac{1}{4}}\cdot k^{\frac{\alpha}{2}-\frac{1}{4}}\cdot e^{2\sqrt{-kx}}\left(1+\mathcal{O}\left(\displaystyle\frac{1}{\sqrt{k}}\right)\right)\ \ \ \ (k\to \infty).$$
Hence we obtain that
$$Q_{2k}(x)=L^{(0)}_{k}(-x)=\displaystyle\frac{1}{2\sqrt{\pi}}e^{2\sqrt{kx}-\frac{x}{2}}(kx)^{-\frac{1}{4}}\left(1+\mathcal{O}\left(\displaystyle\frac{1}{\sqrt{k}}\right)\right)\ \ \ \ (k\to \infty)$$
and
$$Q_{2k-1}(x)=kL^{(-1)}_{k}(-x)=\displaystyle\frac{1}{2\sqrt{\pi}}e^{2\sqrt{kx}-\frac{x}{2}}(kx)^{\frac{1}{4}}\left(1+\mathcal{O}\left(\displaystyle\frac{1}{\sqrt{k}}\right)\right)\ \ \ \ (k\to \infty).$$
Then we obtain the following Proposition.
\begin{prop}
We have the following aymptotic relation between $Q_{2k}(x)$ and $Q_{2k-1}(x)$:
$$Q_{2k-1}(x)\sim \sqrt{kx}Q_{2k}(x)\ \ \ \ (k\to \infty).$$
\end{prop}
\section{An application to irrationality of Euler-Gompertz constant}
\ \ \ \ Euler-Gompertz constant $\delta$ is defined by $\displaystyle\int_{0}^{\infty} \displaystyle\frac{e^{-t}}{t+1}dt$. As far as we know, irrationality of $\delta$ is still an open problem. We see that $\delta=F(1)$, so for any integer $k\ge 1$,
$$\displaystyle\frac{P_{2k}(1)}{Q_{2k}(1)}<\delta<\displaystyle\frac{P_{2k-1}(1)}{Q_{2k-1}(1)}.$$
In the following for the sake of simplicity, we denote $P_{m}(1)$ and $Q_{m}(1)$ by $p_{m}$ and $q_{m}$, respectively. Then we obtain the following inequality:
$$\left|\delta-\displaystyle\frac{p_{2k}}{q_{2k}}\right|< \displaystyle\frac{1}{q_{2k}q_{2k-1}}.$$
From equalities (4.1), (4.2), (4.3), and (4.4), we easily see that $\left[\displaystyle\frac{m}{2}\right]!p_{m}$ and $\left[\displaystyle\frac{m}{2}\right]!q_{m}$ are positive integer. Then multiplying both sides by $k!q_{2k}$, we have that
$$\left|k!q_{2k}\delta-k!p_{2k}\right|<\displaystyle\frac{k!}{q_{2k-1}}.$$
Hence if the right hand side approaches 0 as $k$ approaches infinity, we have that Euler-Gompertz constant $\delta$ is irrational, but unfortunately the right hand side approaches infinity as $k$ approaches infinity, since 
$$q_{2k-1}=\displaystyle\sum_{i=0}^{k-1}\displaystyle\frac{ {}_{k} C _{i+1}}{i!}<2^{k}.$$
(A numerical calculation suggests that $\left|k!q_{2k}\delta-k!p_{2k}\right|$ diverges to infinity). So we consider integer parts of $p_{m}$ and $q_{m}$. Then we obtain the following theorem.
\begin{thm}
We have that
$$\lim_{m\to \infty} \displaystyle\frac{[p_{m}]}{[q_{m}]}=\delta.$$
\end{thm}
\begin{proof}
We easily see the following inequalities:
$$p_{m}-1<[p_{m}]\le p_{m},\ \ \ \ q_{m}-1<[q_{m}]\le q_{m}.$$
Therefore, we have that 
$$\displaystyle\frac{p_{m}-1}{q_{m}}<\displaystyle\frac{[p_{m}]}{[q_{m}]}<\displaystyle\frac{p_{m}}{q_{m}-1}.$$
Since the leftmost side equals $\displaystyle\frac{p_{m}}{q_{m}}-\displaystyle\frac{1}{q_{m}}$, we obtain that the leftmost side approches $\delta$ as $m$ approches infinity. Similarly since the rightmost side equals $\displaystyle\frac{p_{m}}{q_{m}}\cdot \displaystyle\frac{q_{m}}{q_{m}-1}$, we have that the rightmost side approches $\delta$ as $m$ approches infinity. Therefore, we see that
$$\lim_{m\to \infty} \displaystyle\frac{[p_{m}]}{[q_{m}]}=\delta.$$
\end{proof}
\vspace{5mm}
We investigate the behavior of $\delta[q_{m}]-[p_{m}]$ as $m$ approches infinity. We have that
$$\delta [q_{m}]-[p_{m}]=\delta q_{m}-p_{m}-(\delta \{q_{m}\}-\{p_{m}\}),$$
where we denote the fractional part of $p_{m}$ (resp. $q_{m}$) by $\{p_{m}\}$ (resp. $\{q_{m}\}$). Since the behavior of $\delta q_{m}-p_{m}$ is well known, our purpose is reduced to investigating that of $\delta \{q_{m}\}-\{p_{m}\}$. Using Mathematica, we have the following figure with respect to the distribution of $\delta \{q_{m}\}-\{p_{m}\}$ for $1\le m\le10000$.
\newpage
\begin{figure}[htbp]
\begin{center}
\includegraphics[width=0.7\linewidth]{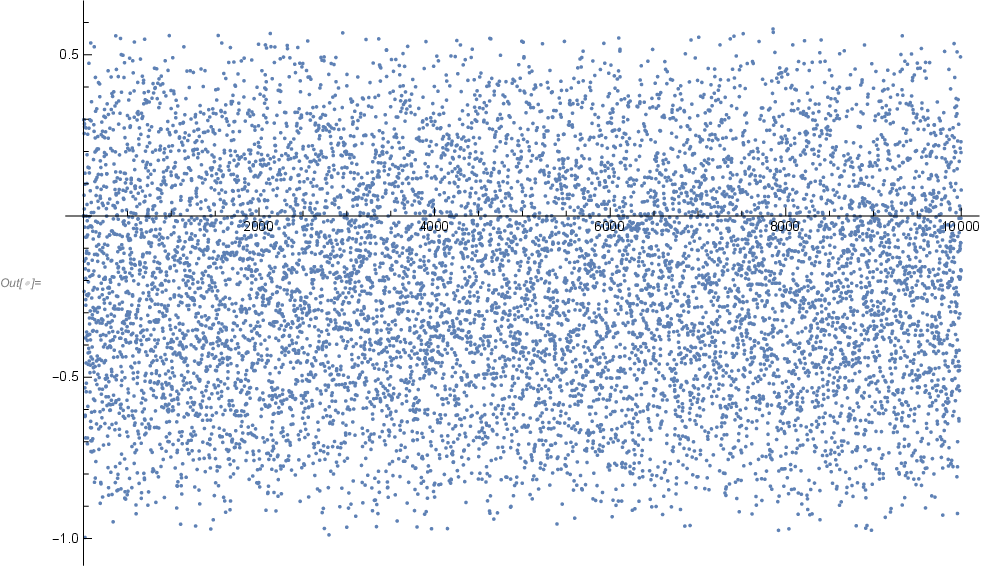}
\end{center}
\end{figure}
From this we read that they are distributed with uniformly except around both ends and can not find any rules that would be satisfied by them. But since the sequence $\delta \{q_{m}\}-\{p_{m}\}\ (m=1,2,3,\cdots)$ is bounded, we can take a convergent subsequence, i.e., positive integers $m_{i}\ \ (i=1,2,3,\cdots)$ and a real number $\alpha$ such that $1\le m_{1}<m_{2}<m_{3}<\cdots<m_{n}<\cdots$ and 
$$\delta\{q_{m_{i}}\}-\{p_{m_{i}}\}\to \alpha\ \ \ \ (i\to \infty).$$
Since $\delta q_{m_{i}}-p_{m_{i}}\to 0\ \ (i\to \infty)$, then we obtain that
$$\delta[q_{m_{i}}]-[p_{m_{i}}]=\delta q_{m_{i}}-p_{m_{i}}-(\delta \{q_{m_{i}}\}-\{p_{m_{i}}\})\to -\alpha\ \ \ \ (i\to \infty).$$
The sequences $p_{m}$ and $q_{m}$ $(m=1,2,3,\cdots)$ are strictly monotonic incresings. Therefore, taking a subsequence of $m_{i}$ $(i=1,2,3,\cdots)$ (which will be denoted by the same symbol), we may assume that the sequences $[p_{m_{i}}]$ and $[q_{m_{i}}]$ $(i=1,2,3,\cdots)$ are strictly monotonic increasings.
Here we put
$$A_{i}:=[p_{m_{i+1}}]-[p_{m_{i}}],\ \ \ \ B_{i}:=[q_{m_{i+1}}]-[q_{m_{i}}]$$
for any integer $i\ge 1$. Hence we obtain that
$$\delta B_{i}-A_{i}=\delta[q_{m_{i+1}}]-[p_{m_{i+1}}]-(\delta[q_{m_{i}}]-[p_{m_{i}}])\to -\alpha+\alpha=0\ \ \ \ (i\to \infty).$$
If $\delta \neq \displaystyle\frac{A_{i}}{B_{i}}$ for any positive integer $i$, we see that Euler-Gompertz constant is irrational by the following theorem.
\begin{thm}[{\cite[Theorem 1.5]{a1}}]
Let $\alpha$ be a real number. We assume that there exists a sequence $\displaystyle\frac{P_{n}}{Q_{n}}$ of rational numbers (where $P_{n}$ and $Q_{n}$ are integers and not necessary to be coprime each other) satisfying
$$0<\left|\alpha-\displaystyle\frac{P_{n}}{Q_{n}}\right|<\displaystyle\frac{\varepsilon(n)}{Q_{n}}$$
for any positive integer $n$, with $\displaystyle\lim_{n\to \infty}\varepsilon(n)=0$. Then, a real number $\alpha$ is irrational.
\end{thm}
\vspace{5mm}
If $\delta \neq \displaystyle\frac{A_{i}}{B_{i}}$ for any positive integer $i$, we have that 
$$0<\left|\delta-\displaystyle\frac{A_{i}}{B_{i}}\right|B_{i}.$$
Since the sequence $\delta B_{i}-A_{i}$ approaches 0 as $i$ approaches infinity, Theorem 7.2 implies that Euler-Gompertz constant $\delta$ is irrational.\\
\ \ \ \ Lastly, we introduce a sufficient condition for Euler-Gompertz constant $\delta$ to be irrational. For this, we prepare the following lemma.
\begin{lem}
For any positive integer $k$, it holds that
$$\delta q_{2k}-p_{2k}<\delta q_{2k-2}-p_{2k-2}$$
and
$$\delta q_{2k+1}-p_{2k+1}>\delta q_{2k-1}-p_{2k-1}.$$
\end{lem}
\begin{proof}
By the inequality $\delta<\displaystyle\frac{p_{2k-1}}{q_{2k-1}}$, we see that $\delta q_{2k-1}-p_{2k-1}<0$. Since $p_{2k}=\displaystyle\frac{1}{k}p_{2k-1}+p_{2k-2}$ and $q_{2k}=\displaystyle\frac{1}{k}q_{2k-1}+q_{2k-2}$. Then we obtain that
$$\delta(kq_{2k}-kq_{2k-2})-(kp_{2k}-kp_{2k-2})<0.$$
Therefore, we have that
$$\delta q_{2k}-p_{2k}<\delta q_{2k-2}-p_{2k-2}$$
for any positive integer $k$.\\
\ \ \ \ Similarly, we obtain the second inequality.
\end{proof}
\vspace{5mm}
Then we can state the following theorem about irrationality of $\delta$.
\begin{thm}
If there exists a subsequence of the sequence $\delta\{q_{2k}\}-\{p_{2k}\}\ (k=1,2,3,\cdots)$ which is monotonic non-decrease or that of the sequence $\delta\{q_{2k-1}\}-\{p_{2k-1}\}\ (k=1,2,3,\cdots)$ which is monotonic non-increase, then Euler-Gompertz constant $\delta$ is irrational.
\end{thm}
\begin{proof}
Firstly, suppose that there exists a subsequence $\delta\{q_{2k_{i}}\}-\{p_{2k_{i}}\}\ (i=1,2,3,\cdots)$ which is monotonic non-decrease, i.e., $k_{i}\ (i=1,2,3,\cdots)$ are positive integers with $1\le k_{1}<k_{2}<k_{3}<\cdots<k_{n}<\cdots$ and 
$$\delta\{q_{2k_{1}}\}-\{p_{2k_{1}}\}\le \delta\{q_{2k_{2}}\}-\{p_{2k_{2}}\}\le \delta\{q_{2k_{3}}\}-\{p_{2k_{3}}\}\le \cdots \le \delta\{q_{2k_{n}}\}-\{p_{2k_{n}}\}\le \cdots.$$
Taking a subsequence of $k_{i}\ (i=1,2,3,\cdots)$\ (whcih will be denoted by the same symbol) if necessary, we may assume that the sequences $[p_{2k_{i}}]$ and $[q_{2k_{i}}]$ $(i=1,2,3,\cdots)$ are strictly monotonic increasings. Since the sequence $\delta\{q_{2k_{i}}\}-\{p_{2k_{i}}\}$ converges as $i$ approachs infinity, the above argument implies that
$$\delta B_{i}-A_{i}\to 0\ \ \ \ (i\to \infty),$$
where $A_{i}$ and $B_{i}$ are the integers defined as above.\\
\ \ \ \ We assume that there exists a positive integer $i$ such that $\delta=\displaystyle\frac{A_{i}}{B_{i}}$. Then we have that
$$\delta [q_{2k_{i+1}}]-[p_{2k_{i+1}}]-(\delta [q_{2k_{i}}]-[p_{2k_{i}}])=0.$$
This is equivalent to
$$\delta q_{2k_{i+1}}-p_{2k_{i+1}}-(\delta q_{2k_{i}}-p_{2k_{i}})=\delta \{q_{2k_{i+1}}\}-\{p_{2k_{i+1}}\}-(\delta \{q_{2k_{i}}\}-\{p_{2k_{i}}\}).\eqno(7.1)$$
By Lemma 7.3, the left hand side of (7.1) is negative, but by the assumption, the right hand side of (7.1) is greater than or equal to zero, which is a contradiction. By Theorem 7.2, we have that Euler-Gompertz constant $\delta$ is irrational.\\
\ \ \ \ In the second case, we similarly get the claim.
\end{proof}
\vspace{5mm}
\newpage
Using Mathematica, we have the following figure with respect to the distribution of $\delta\{q_{2k}\}-\{p_{2k}\}$ for $1\le k\le 5000$.

\begin{figure}[htbp]
\begin{center}
\includegraphics[width=0.7\linewidth]{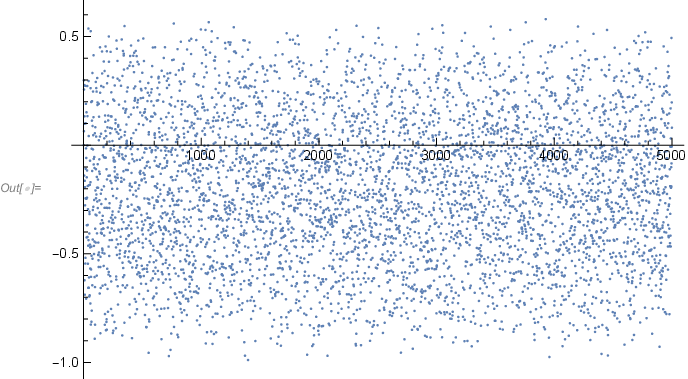}
\end{center}
\end{figure}
From this we see that the distribution is close to a uniform distribution (except around both ends). Therefore, we can expect existance of a desired subsequence. A similar observation is made in the sequence $\delta\{q_{2k-1}\}-\{p_{2k-1}\}\ (k=1,2,3,\cdots)$.
\begin{con}
There exists a subsequences of the sequence $\delta\{q_{2k}\}-\{p_{2k}\}\ (k=1,2,3,\cdots)$ which is monotonic non-decrease and that of the sequence $\delta\{q_{2k-1}\}-\{p_{2k-1}\}\ (k=1,2,3,\cdots)$ which is monotonic non-increase.
\end{con}

\vspace{5mm}
{\bf Address:}  Naoki Murabayashi: Department of Mathematics, Faculty of Engineering Science, Kansai \\
\hspace{10mm}University, 3-3-35, Yamate-cho, Suita-shi, Osaka, 564-8680, Japan.\\
{\bf E-mail:} murabaya@kansai-u.ac.jp \\ \\
{\bf Address:}  Hayato Yoshida: Mathematics, Integrated Science and Engineering Major, Graduate school of Science and\\
\hspace{10mm}Engineering, Kansai University, 3-3-35, Yamate-cho, Suita-shi, Osaka, 564-8680, Japan.\\ 
{\bf E-mail:} k321930@kansai-u.ac.jp
\end{document}